\documentclass[12pt]{amsart}

\usepackage{graphicx}
\usepackage{amsthm}
\usepackage{amsmath, amssymb}
\swapnumbers

\theoremstyle{plain}
\newtheorem{Thm}{Theorem}

\newtheorem{Coro}[Thm]{Corollary}
\newtheorem{Lem}[Thm]{Lemma}

\theoremstyle{definition}
\newtheorem{Def}[Thm]{Definition}

\begin{document}

\title[Mapping class groups of Heegaard splittings]{Mapping class groups of medium distance Heegaard splittings}

\author{Jesse Johnson}
\address{\hskip-\parindent
        Department of Mathematics \\
        Oklahoma State University \\
        Stillwater, OK 74078
        USA}
\email{jjohnson@math.okstate.edu}


\thanks{Research supported by NSF MSPRF grant 0602368}

\begin{abstract}
We show that if the Hempel distance of a Heegaard splitting is larger than three then the mapping class group of the Heegaard splitting is isomorphic to a subgroup of the mapping class group of the ambient 3-manifold.  This implies that given two handlebody sets in the curve complex for a surface that are distance at least four apart, the group of automorphisms of the curve complex that preserve both handlebody sets is finite.
\end{abstract}

\maketitle

\section{Introduction}

A \textit{Heegaard splitting} for a compact, connected, closed, orientable 3-manifold $M$ is a triple $(\Sigma, H^-, H^+)$ where $\Sigma$ is a compact, separating surface in $M$ and $H^-$, $H^+$ are handlebodies in $M$ such that $M = H^- \cup H^+$ and $\partial H^- = \Sigma = H^- \cap H^+ = \partial H^+$.  The \textit{automorphism group} $Aut(M, \Sigma)$ of the Heegaard splitting is the set of automorphisms of $M$ that take $\Sigma$ onto itself.  The \textit{mapping class group} $Mod(M,\Sigma)$ of the Heegaard splitting is the group of connected components of $Aut(M, \Sigma)$.  Because every automorphism of $(M, \Sigma)$ is an automorphism of $M$, there is a canonical map $i : Mod(M, \Sigma) \rightarrow Mod(M)$.

The \textit{curve complex} $C(\Sigma)$ is the simplicial complex whose vertices are isotopy classes of essential simple closed curves in $\Sigma$, with edges connecting disjoint loops.  
The \textit{Hempel distance} $d(\Sigma)$ of $\Sigma$ is the distance in the curve complex from the set of loops bounding disks in one handlebody to the set of loops bounding disks in the other.  We prove the following:

\begin{Thm}
\label{mainthm}
Let $(\Sigma, H^-, H^+)$ be a genus $k > 1$ Heegaard splitting of a 3-manifold $M$.  If $d(\Sigma) > 3$ then $Mod(M, \Sigma)$ is finite and the induced homomorphism $i$ is an injection. If $d(\Sigma) > 2k$ then $i$ is an isomorphism.
\end{Thm}

This theorem improves a result of Namazi~\cite{nam:mcg} showing that for some constant $C_k$ depending on the genus $k$ of $\Sigma$, if $d(\Sigma) > C_k$ then $Mod(M,\Sigma)$ is finite.  The author and Hyam Rubinstein~\cite{finitemcg} showed that this implies the homomorphism $i$ is an isomorphism.  We also showed that a Heegaard splitting with Hempel distance greater than four cannot have an infinite order, reducible automorphism.  

Examples of distance two Heegaard splittings admitting infinite order reducible automorphisms are known to exist; we noted in~\cite{finitemcg} that Heegaard splittings induced by open book decompositions have such automorphisms, many of which have distance two.  Long~\cite{long:twoh} found strongly irreducible Heegaard splittings that admit pseudo-Anosov automorphism.  These latter examples have distance at least two, though to the author's knowledge, the exact distance of Long's example is not known.  Thus it remains an open problem whether there is a Heegaard splitting with Hempel distance three that admits an infinite order automorphism.

The main thrust of this paper will be to show that for distance greater than three, the homomorphism $i$ is injective.  Hempel~\cite{hempel} and Thompson~\cite{thompson} showed that any 3-manifold admitting a Heegaard splitting of genus at least three is neither toroidal nor Seifert fibered, so the geometrization theorem implies that every such 3-manifold is hyperbolic, and thus has a finite mapping class group.  We thus conclude that every Heegaard splitting with Hempel distance at least 4 has a finite mapping class group.  This implies the following Corollary for the curve complex:

\begin{Coro}
\label{curvecomplexcoro}
If $\mathcal{H}$ and $\mathcal{H}'$ are handlebody sets in the curve complex $C(\Sigma)$ and $d(\mathcal{H}, \mathcal{H}') > 3$ then the group of automorphisms of $C(\Sigma)$ that take each of $\mathcal{H}$, $\mathcal{H}'$ onto itself is finite.
\end{Coro}

The proof of Theorem~\ref{mainthm} is based on a refinement of sweep-out based techniques introduced by the author in~\cite{me:stabs}, with a better distance bound due to the fact that we are comparing a Heegaard splitting to itself.  Section~\ref{sweepsect} introduces sweep-outs and graphics.  The ``spanning'' and ``splitting'' conditions introduced in~\cite{me:stabs} and the ways these two conditions change during an isotopy of the sweep-out are discussed in Section~\ref{facingsect}.  (Any element of the kernel of $i$ implies such an isotopy.)  We show in Section~\ref{projectsect} that if the graphic remains spanning throughout the isotopy then the induced element of the kernel is trivial.  We show in section~\ref{curvesect} that if the graphic changes from spanning to splitting during the isotopy then the distance of the Heegaard splitting is at most three.  We then combine these results in Section~\ref{mainthmsect} to prove Theorem~\ref{mainthm} and Corollary~\ref{curvecomplexcoro}.  
I want to thank Yoav Moriah for many helpful discussions.

\section{Sweep-outs and graphics}
\label{sweepsect}

A \textit{handlebody} is a connected 3-manifold that is homeomorphic to a regular neighborhood of a graph $K$ embedded in $S^3$.  The graph $K$ is called a \textit{spine} for $H$.

A \textit{sweep-out} is a smooth function $f : M \rightarrow [-1,1]$ such that for each $x \in (-1,1)$, the level set $f^{-1}(x)$ is a closed surface.  Moreover, $f^{-1}(-1)$ must be the union of a graph and some number of boundary components while $f^{-1}(1)$ is the union of a second graph and the remaining boundary components.  Each of $f^{-1}(-1)$ and $f^{-1}(1)$ is called a \textit{spine} of the sweep-out.  Each level surface of $f$ is a Heegaard surface for $M$ and the spines of the sweep-outs are spines of the two handlebodies in the Heegaard splitting.  

Conversely, given a Heegaard splitting $(\Sigma, H^-, H^+)$ for $M$, there is a sweep-out for $M$ such that each level surface is isotopic to $\Sigma$.  We will say that a sweep-out \textit{represents} $(\Sigma, H^-, H^+)$ if $f^{-1}(-1)$ is isotopic to a spine for $H^-$ and $f^{-1}(1)$ is isotopic to a spine for $H^+$.  The level surfaces of such a sweep-out will be isotopic to $\Sigma$.  Because the complement $H \setminus K$ is a surface cross an interval, we can construct a sweep-out for any Heegaard splitting, i.e. we have the following:

\begin{Lem}
Every Heegaard splitting of a compact, connected, closed orientable, smooth 3-manifold is represented by a sweep-out.
\end{Lem}

Given two sweep-outs, $f$ and $g$, their product is a smooth function $f \times g : M \rightarrow [-1,1] \times [-1,1]$.  (That is, we define $(f \times g)(x) = (f(x),g(x))$.)  The \textit{discriminant set} for $f \times g$ is the set of points where the level sets of the two functions are tangent.  Kobayashi~\cite{Kob:disc} has shown that after an isotopy of $f$ and $g$, we can assume that $f \times g$ is a stable function on the complement of the four spines.  This implies that the discriminant set will be a one dimensional smooth submanifold in the complement in $M$ of the spines.

The function $f \times g$ sends the discriminant to a graph in $[-1,1] \times [-1,1]$ called the \textit{Rubinstein-Scharlemann graphic} (or just the \textit{graphic} for short).  Kobayashi's approach uses singularity theory to recover the machinery originally constructed by Rubinstein and Scharlemann in ~\cite{rub:compar} using Cerf theory.

\begin{Def}
The function $f \times g$ is \textit{generic} if $f \times g$ is stable and each arc $\{t\} \times [-1,1]$ or $[-1,1] \times \{s\}$ contains at most one vertex of the graphic.
\end{Def}

\section{Spanning and Splitting}
\label{facingsect}

Let $f$ and $g$ be sweep-outs representing the same Heegaard splitting $(\Sigma, H^-, H^+)$.  For each $s \in (-1,1)$, define $\Sigma'_s = g^{-1}(s)$, $H'^-_s = g^{-1}([-1,s])$ and $H'^+_s = g^{-1}([s,1])$.  Similarly, for $t \in (-1,1)$, define $\Sigma_t = f^{-1}(t)$.  We will say that $\Sigma_t$ is \textit{mostly above} $\Sigma'_s$ if each component of $\Sigma_t \cap H'^-_s$ is contained in a disk subset of $\Sigma_t$.  Similarly, $\Sigma_t$ is \textit{mostly below} $g_t$ if each component of $\Sigma_t \cap H'^+_s$ is contained in a disk in $\Sigma_t$.  See~\cite{me:stabs} for a more detailed discussion of these definitions.

Let $R_a \subset (-1,1) \times (-1,1)$ be the set of all values $(t,s)$ such that $\Sigma_t$ is mostly above $\Sigma'_s$.  Let $R_b \subset (-1,1) \times (-1,1)$ be the set of all values $(t,s)$ such that $\Sigma_t$ is mostly below $\Sigma'_s$.  For a fixed $t$, there will be values $a,b$ such that $\Sigma_t$ will be mostly above $\Sigma'_s$ if and only if $s \in [-1,a)$ and mostly above $\Sigma'_s$ if and only if $s \in [b,1]$, so both regions will be vertically convex.

\begin{Def}
\label{ffdef}
Given a generic pair $f$, $g$, we will say $g$ \textit{spans} $f$ if there is a horizontal arc $\{s\} \times [-1,1]$ that intersects the interiors of both regions $R_a$ and $R_b$.  Otherwise, we will say that $g$ \textit{splits} $f$.
\end{Def}

We noted in~\cite{me:stabs} the following:

\begin{Lem}
\label{regionsboundedlem}
The closure of $R_a$ in $(-1,1) \times (-1,1)$ is bounded by arcs of the Rubinstein-Scharlemann graphic, as is the closure of $R_b$.  The closures of $R_a$ and $R_b$ are disjoint.
\end{Lem}

We will rely on the following Lemma, which follows almost immediately from the definition of spanning.  The details are left to the reader:

\begin{Lem}
\label{withitselflem}
If a level surface of $g$ is disjoint from and separates the spines of $f$ then $g$ spans $f$.
\end{Lem}

The \textit{ambiently trivial subgroup} $Amt(M, \Sigma)$ is the kernel of the induced homomorphism $i : Mod(M, \Sigma) \rightarrow Mod(M)$ defined in the introduction.  An element $\phi \in Amt(M, \Sigma)$ is a homeomorphism from $M$ to itself that is isotopy trivial on $M$ but restricts to a non-trivial self-homeomorphism of $\Sigma$.  Let $\{\Phi_r : M \rightarrow M | r \in [0,1]\}$ be an isotopy such that $\Phi_0$ is the identity and $\Phi_1 = \phi$.

Let $f$ be a sweep-out representing $(\Sigma, H^-, H^+)$ and let $g$ be the result of perturbing $f$ slightly so that $f \times g$ is generic.  Because every level set of $f$ separates the spines of $f$, we can choose $g$ so that some level set $\Sigma'_s$ of $g$ separates the spines of $f$.  Thus by Lemma~\ref{withitselflem}, the sweep-out $g$ will span $f$.

Because $g$ is isotopic to $f$ and $f$ represents $(\Sigma, H^-, H^+)$, the level surface $\Sigma'_s$ is isotopic to $\Sigma$ and this isotopy extends to an ambient isotopy of $M$ sending $\Sigma'_s$ onto $\Sigma$.  Conjugate both $f$ and $g$ by this ambient isotopy so that $f \times g$ is generic, $g$ spans $f$ and $\Sigma'_s = \Sigma$ separates the spines of $f$.

Because $\phi(\Sigma) = \Sigma$, the sweep-out $g \circ \phi$ will also span $f$.  Define the family of sweep-outs $g_r = g \circ \Phi_r$.  The following Lemma is proved in~\cite{me:stabs}:

\begin{Lem}[Lemma 25 in~\cite{me:stabs}]
\label{isotopesweepslem}
We can choose the isotopy $\{\Phi_r\}$ so that the graphic defined by $f$ and $g_r$ is generic for all but finitely many points.  At the finitely many non-generic points, there are at most two valence two or four vertices at the same level, or one valence six vertex.
\end{Lem}

Away from the finitely many non-generic points, the sweep-out $g_r$ either spans or splits $f$.  We will consider two cases:  When every $g_r$ spans $f$ and when some $g_r$ fails to span $f$.

\begin{Lem}
\label{allspanninglem}
If every generic $g_r$ spans $f$ then the restriction of $\phi$ to $\Sigma$ is isotopy trivial on $\Sigma$.
\end{Lem}

\begin{Lem}
\label{somesplitlem}
If some generic $g_r$ splits $f$ then the $d(\Sigma) \leq 3$.
\end{Lem}

We will prove Lemma~\ref{allspanninglem} in Section~\ref{projectsect} and prove Lemma~\ref{somesplitlem} in Section~\ref{curvesect}.

\section{Projections}
\label{projectsect}

\begin{proof}[Proof of Lemma~\ref{allspanninglem}]
For each $r$ such that $g_r$ spans $f$, we will define a projection from the surface $g^{-1}(0)$ onto $f^{-1}(0)$.

Because $g_r$ spans $f$, there is a value $s \in [-1,1]$ such that the arc $\{s\} \times [-1,1]$ passes through both regions $R_a$ and $R_b$.  Thus there are values $a$, $b$ such that $\Sigma_a$ is mostly above $\Sigma'_s$ and $\Sigma_b$ is mostly below $\Sigma'_s$.  Every loop of intersection between $\Sigma'_s$ and $\Sigma_a \cup \Sigma_b$ is trivial in $\Sigma_a$ or $\Sigma_b$, so we can compress along innermost loops until we have formed a surface $S$ disjoint from both $\Sigma_a$ and $\Sigma_b$.  

Because $\Sigma_a$ was mostly above $\Sigma'_s$ and $\Sigma_b$ is mostly below, the surface $S$ will separate $\Sigma_a$ from $\Sigma_b$.  Thus, as in~\cite{me:stabs} the surface $S$ has genus at least $k$, where $k$ is the genus of $\Sigma$.  Because $\Sigma'_s$ has genus $k$ and $S$ is formed from $\Sigma'_s$ by a series of compressions, the compressions must all be along loops that are trivial in $\Sigma'_s$.

The surface $S$ coincides with $\Sigma'_s$ in the complement of a neighborhood of the disks bounding these trivial compressions.  Any homeomorphism between the boundaries of two disks extends into the interiors of the disks, and the induced homeomorphism of the disks is uniquely determined up to isotopy.  Thus the identity map on the set where $S$ and $\Sigma'_s$ coincide extends to a homeomorphism from $\Sigma'_S$ onto $S$ and this homeomorphism is well defined up to isotopy.

The surface $S$ is contained in the surface-cross-interval between $\Sigma_a$ and $\Sigma_b$ so there is, up to isotopy, a canonical projection of $S$ onto $f^{-1}(0)$.  Composing this with the projection from $g^{-1}(0)$ onto $\Sigma'_s$ and the projection from $\Sigma'_s$ onto $S$, we find a projection from $g^{-1}(0)$ onto $f^{-1}(0)$ that is uniquely determined up to isotopy.

The projection map varies continuously for different values of $s$, as long as those values stay in the spanning range.  As long as $g_r$ is spanning, the projection also varies continuously with $r$, so it remains in the same isotopy class.  The initial projection map is the identity, so if $g_r$ spans $f$ for all generic values of $r$ then $\phi$ is the identity.
\end{proof}

\section{Curves}
\label{curvesect}

\begin{proof}[Proof of Lemma~\ref{somesplitlem}]
The sweep-outs $g_0$ and $g_1$ both span $f$, so if some $g_r$ splits $f$ then there is a non-generic value $r_0$ with generic values $r_-$, $r_+$, just before and after $r_0$, respectively, such that $g_{r_-}$ spans $f$ while $g_{r_+}$ splits $f$. 

Let $A_r$ be the projection of the set $R_a \subset [-1,1] \times [-1,1]$ for $f \times g_r$ onto the $t$-axis of the square.  Let $B_r$ be the projection of $R_b$ at time $r$.  For generic $r$, the sets $A_r$ and $B_r$ are of the form $(-1,s_a)$ and $(s_b, 1)$, respectively.  The endpoints $s_a$, $s_b$ are projections of vertices of the graphic.  The sweep-out  $g_r$ will span $f$ if and only if the sets $A_r$, $B_r$ have non-trivial intersection, i.e. when $s_a < s_b$.
\begin{figure}[htb]
  \begin{center}
  \includegraphics[width=3.5in]{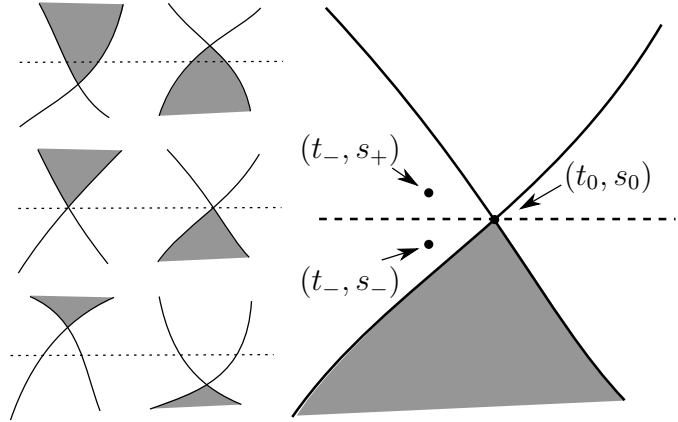}
  \put(-43,90){$(t_0, s_0)$}
  \put(-143,100){$(t_-, s_+)$}
  \put(-143,50){$(t_-, s_-)$}
  \caption{The graphic changes from spanning to splitting.}
  \label{grfig}
  \end{center}
\end{figure}

As the parameter $r$ passes from $r_-$ to $r_+$, the sets $A_r$, $B_r$ initially intersect in an open interval.  At time $r_-$, the intersection is still an open interval, but at time $r_+$, the intersection is empty.  Thus at time $r_0$, the two projections must be disjoint such that their complement contains a single point $s_0 = s_a = s_b$.  At any time $r$, the horizontal arcs $[-1,1] \times s_a$ and $[-1,1] \times s_b$ pass through vertices of the graphic, each in the closure of $R_a$ or $R_b$, respectively.  At time $r_0$, the arc $[-1,1] \times s_0$ intersects two vertices, one of which is in the closure of $R_a$ and the other in the closure of $R_b$.  The pictures at times $r_-$, $r_0$ and $r_+$ are shown on the left side of Figure~\ref{grfig}.

Let $t_a$ and $t_b$ be values such that the two vertices of the graphic have coordinates $(t_a, s_0)$ and $(t_b, s_0)$.  A blowup of the point $(t_a, s_0)$ is shown on the right in Figure~\ref{grfig}.  Because $g$ initially spans $f$ positively, we must have that $t_b < t_a$.  By the same argument as in Lemma~\ref{allspanninglem}, the projection map of the subsurface $g^{-1}(s_0) \cap f^{-1}(t_b + \epsilon, t_a - \epsilon)$ can be extended to a homotopy equivalence between the two surfaces, so any two disjoint essential loops in this subsurface project to (isotopy clases of) disjoint loops in $\Sigma$.  This is true, in particular, for any essential loops in the level sets of $f|_{g^{-1}(s_0)}$ at levels $t_b + \epsilon$ and $t_a - \epsilon$.  We will consider below the level set at level $t_a - \epsilon$, but a symmetric argument applies to the level set at time $t_b + \epsilon$.

Choose $\epsilon > 0$, then choose $\delta > 0$ so that $(t_a - \epsilon, s_0 \pm \delta)$ are in the same component of the complement of the graphic as $(t_a - \epsilon, s_0)$.  Define $t_- = t_a - \epsilon$, $s_\pm = s_0 \pm \delta$ and define the functions $f_\pm$, $f_0$ on $\Sigma'_{s_\pm}$, $\Sigma'_{s_0}$ as the restrictions of $f$ to these surfaces.  Because the point $(t_-, s_-)$ is outside $R_a$ and $R_b$, the level set of $f_+$ at level $s_+$ contains an essential loop in $\Sigma_{s_-}$.  The point $(t_a, s_-)$, on the other hand, is in $R_b$ so every loop in the level set of $f_-$ at level $t_a$ is trivial in $\Sigma'_{s_-}$.  

There is a single critical point between levels $t_-$, $t_0$ in the Morse function $f_-$, so this critical point must be a saddle and as we pass from $t_-$ to $t_a$, a pair of parallel essential loops must be pinched together to produce a single trivial loop in $\Sigma'_{s_-}$.  Thus the image in $\Sigma'_{s_-}$ of $f_-^{-1}([t_-,1])$ consists of an annulus and a collection of disks, which are properly embedded in the handlebody $H^+_{t_-} = f^{-1}([t_-,1])$.  

Because the points $(t_-, s_+)$ and $(t_-, s_-)$ are in the same component of the complement of the graphic, the level sets of $f_+$ and $f_-$ at level $t_-$ are isotopic in $\Sigma_{t_-}$ and project to isotopic collections of curves in $\Sigma'$.  Thus in the Morse function $f_+$, the level loops at level $s_-$ bound an annulus and a collection of disks in $\Sigma'_{s_+}$ that are properly embedded in $H^+_{t_+}$ and isotopic to the annulus and disks defined by $\Sigma'_{s_-}$.  

Because no point $(t, s_+)$ with $t > t_a$ is in $R_a$ or $R_b$, every level set of $f_+$ at a level $t > t_a$ contains a loop that is essential in $\Sigma_t$.  If this essential loop bounds a disk in $\Sigma'_t$ then it bounds compressing disk in the handlebody $H^+_t$.  There is a saddle in $f_+$ that pinches the annulus to produce a disk.  This saddle either produces an essential disk, in which case the boundary of the annulus is disjoint from this disk, or it produces a trivial disk, in which case some other loop in the level set bounds a compressing disk in $H^+_t$.  In either case, the annulus is disjoint from a compressing disk and is isotopic to the original annulus whose boundary is contained in the level set of $f_0$ at times $t_-$.  

We have shown that an essential level loop of $f|_{g^{-1}(s_0)}$ at level $t_a - \epsilon$ is distance one from a compressing disk in $H^+$.  A symmetric argument shows that an essential level loop of $f|_{g^{-1}(s_0)}$ at level $t_b + \epsilon$ is distance one from a compressing disk in $H^-$.  As noted above, these level sets are disjoint from each other in $\Sigma$, so the distance of $\Sigma$ is at most 3.
\end{proof}

\section{Proof of the Main Theorem}
\label{mainthmsect}

\begin{proof}[Proof of Theorem~\ref{mainthm}]
Assume there is a non-trivial element $\phi$ of the kernel of the homomorphism $i : Mod(M, \Sigma) \rightarrow Mod(M)$.  By Lemma~\ref{isotopesweepslem}, we can realize $\phi$ by a continuous family of sweep-outs $\{g_r\}$ such that $f \times g_r$ is generic for all but finitely many values of $r$.  If $g_r$ spans $f$ for every generic value of $r$ then by Lemma~\ref{allspanninglem}, the restriction of $\phi$ to $\Sigma$ is isotopy trivial.  Because $\phi$ is a non-trivial element of $Mod(M, \Sigma)$, the restriction must be non-trivial, so there is generic value of $r$ where $g_r$ splits $f$.  By Lemma~\ref{somesplitlem}, this implies that the Hempel distance of $\Sigma$ is at most 3.

Thus if $d(\Sigma) > 3$ then the kernel of $i$ must be trivial, i.e. $i$ is an injection.  By Hempel~\cite{hempel} and Thompson~\cite{thompson}, if $d(\Sigma) > 2$ then $M$ is atoroidal and not Seifert fibered, so $M$ is hyperbolic by the geometrization theorem.  This implies that $Mod(M)$ is finite.  Since the image of $i$ is finite and its kernel is trivial, the domain $Mod(M, \Sigma)$ of $i$ must be finite, completing the proof.

Scharlemann and Tomova~\cite{st:dist} showed that if $d(\Sigma) > 2k$ (where $k$ is the genus of $\Sigma$) then every genus $k$ Heegaard splitting of $M$ is isotopic to $\Sigma$.  This implies that every automorphism of $M$ can be composed with an isotopy so that the automorphism takes $\Sigma$ onto itself.  Thus every element of $Mod(M)$ is represented by an element of $Mod(M, \Sigma)$ and $i$ is onto.  If $k \geq 2$ then $d(\Sigma) \geq 2k > 3$ so $i$ is also one-to-one, and thus $i$ is an isomorphism.
\end{proof}

\begin{proof}[Proof of Corollary~\ref{curvecomplexcoro}]
Any pair of handlbody sets in the curve complex determines a Heegaard splitting $(\Sigma, H^-, H^+)$ in some 3-manifold $M$ whose Hempel distance is (by definition) the distance in the curve complex between the two handlebody sets.  By~\cite{ivanov},~\cite{luo}, every isometry of the curve complex is induced by an automorphism of $\Sigma$.  If an automorphism preseves the handlebody sets then it will take any loop bounding a disk in $H^-$ or $H^+$ to the boundary of another disk in $H^-$ or $H^+$.  Thus the automorphism of $\Sigma$ extends to each of the handlebodies, and thus to all of $M$.  In other words, any isometry of the curve complex determines an element of $Mod(M, \Sigma)$.  By Theorem~\ref{mainthm}, if the handlebodies are distance greater than 3 apart, then $Mod(M, \Sigma)$ is finite, so the group of isometries preserving the two handlebody sets is finite.
\end{proof}

\bibliographystyle{amsplain}
\bibliography{hidismcg}

\end{document}